\documentclass[11pt,a4paper,reqno]{amsart}
\usepackage{latexsym}
\usepackage{bbm,amssymb,amsthm,amsmath}
\usepackage{graphicx}
\usepackage{multirow}
\usepackage{caption}
\usepackage{bm}
\usepackage{multirow}
\usepackage{mathtools}  
\usepackage{float}  

\usepackage{xcolor}
\usepackage{enumitem}

\usepackage{tikz}	
\usetikzlibrary{shapes,arrows}
\usetikzlibrary{positioning}

\calclayout

\theoremstyle{plain}
\newtheorem{theorem}{Theorem}[section]	
\newtheorem{lemma}[theorem]{Lemma}
\newtheorem{corollary}[theorem]{Corollary}

\theoremstyle{definition}
\newtheorem{definition}[theorem]{Definition}
\newtheorem{example}[theorem]{Example}

\theoremstyle{remark}
\newtheorem{remark}[theorem]{Remark}

\numberwithin{equation}{section}	

\def\C{\mathbb{C}}
\def\R{\mathbb{R}}    
  
\def\N{\mathbb{N}}  

\def\Rd{\mathbb{R}^{d}}

\def\Cr{\mathbb{C}^{r}}
\def\Mat#1{{{\rm M}_{#1}(\C)}}

\let\mib=\boldsymbol

\def\mnu{{\mib \nu}}

\def\mx{\mathbf{x}}

\def\vf{\mathsf{f}}
\def\vu{\mathsf{u}}
\def\vv{\mathsf{v}}

\def\vnu{\mathsf{\nu}}

\def\mA{\mathbf{A}}
\def\mP{\mathbf{P}}
\def\mB{\mathbf{B}}

\def\mI{\mathbf{I}}

\def\mR{\mathbf{R}}

\def\lD{\mathcal{D}}
\def\lH{\mathcal{H}}
\def\lV{\mathcal{V}}
\def\lW{\mathcal{W}}
\def\lX{\mathcal{X}}

\def\phi{\mathsf{\varphi}}
\def\eps{\varepsilon}

\def\dom{\operatorname{dom}}
\def\dim{\operatorname{dim}}
\def\ker{\operatorname{ker}}
\def\ran{\operatorname{ran}}

\def\cl{\operatorname{cl}}

\def\rank{\operatorname{rank}}

\def\dup#1#2#3#4{{}_{#1\!}\langle\, #2 , #3 \,\rangle_{#4}} 
\def\scp#1#2{\langle\, #1 \mid #2 \,\rangle}  
\def\iscp#1#2{[\, #1 \mid #2 \,]}  

\begin{document}

\title[Friedrichs systems on an interval]{Friedrichs systems on an interval}

\author{M.~Erceg}\address{Marko Erceg,
	Department of Mathematics, Faculty of Science, University of Zagreb, Bijeni\v{c}ka cesta 30,
	10000 Zagreb, Croatia}\email{maerceg@math.hr}

\author{S.~K.~Soni}\address{Sandeep Kumar Soni,
	Department of Mathematics, Faculty of Science, University of Zagreb, Bijeni\v{c}ka cesta 30,
	10000 Zagreb, Croatia}\email{sandeep@math.hr}




\begin{abstract}
There has been significant developments in the classification of boundary conditions of positive symmetric systems, also known as Friedrichs systems, after the introduction of operator theoretic framework. We take a step forward towards applying the abstract theory to the classical framework by studying Friedrichs systems on an interval. Dealing with some difficulties related to the smoothness of eigenvectors, here we present an explicit expression for the dimensions of the kernels of Friedrichs operators only in terms of the values of the coefficients at the end-points of the interval. In particular, this allows for a characterisation of all admissible boundary conditions, i.e.~those leading to bijective realisations.
\end{abstract}

\maketitle


\section{Overview}\label{sec:Overview}
 
Continuing his research on \emph{symmetric hyperbolic systems} \cite{KOFh}, Friedrichs introduced the concept of \emph{positive symmetric systems} of first order linear partial differential equations \cite{KOF}, which are better known as \emph{Friedrichs systems}. More precisely,
\begin{definition}\label{def:CFO}
For a given open and bounded set $\Omega \subseteq \Rd$ with Lipschitz boundary $\Gamma$, let the matrix functions 
$\mA_k \in W^{1,\infty}(\Omega;\Mat{r})$, $k=1,2,\dots,d$,
and $\mB \in L^\infty(\Omega; \Mat{r})$ satisfy 
\begin{equation}
\mA_k= \mA_k^\ast \qquad \hbox{on} \,\, \Omega
\tag{F1}
\end{equation}
and
\begin{equation}
(\exists{\mu_0>0}) \quad \mB + \mB^\ast+\sum_{k=1}^d 
    \partial_k \mA_k\geq 2\mu_0\mI
    \qquad {\hbox{a.e.~on}\,\, \Omega}\,.
\tag{F2}
\end{equation}
Then the first-order differential operator $L : L^2(\Omega)^r\longrightarrow \mathcal{D}'(\Omega)^r$
defined by
\begin{equation}
L \vu \;:=\; \sum_{k=1}^d \partial_k(\mA_k \vu) + \mB \vu
\tag{CFO}
\end{equation}
(here derivatives are taken in the distributional sense) 
is called \emph{the (classical) Friedrichs operator} or \emph{the symmetric positive operator}, while 
(for given $\vf \in L^2(\Omega)^r$) the first-order system of partial differential
equations $L \vu = \vf$ is called \emph{the (classical) Friedrichs system} or 
\emph{the symmetric positive system}.
\end{definition}

Although the main motivation for introducing this concept was to study differential equations which change their type, like the \emph{Tricomi equation}, which appears in the transonic flow, this setting proved to be successful in studying various differential equations and related problems. Further development of the theory can be traced in the following references \cite{FL, PS}. 

In the classical theory, Friedrichs proposed a clever way of representing different boundary (and/or initial) conditions via a matrix field on the boundary, which we refer to as (FM)-boundary conditions. Two more but equivalent ways of representing (initial-)boundary conditions are studied in \cite{FL} and \cite{PS}, which we refer to as (FX) and (FV)-boundary conditions, respectively. These boundary conditions were well studied in some differential equations, however the non-uniqueness of matrix field on the boundary posed some challenges. Moreover, only the existence of weak solutions and uniqueness of the strong ones were achieved, leaving the general well-posedness problem open. 

One advantage of this theory is in the numerical studies of differential equations.
Indeed, this framework provides a unified theory for all types of (semi)linear partial differential equations, while the structure of first-order equations is beneficial for developing numerical schemes (see e.g.~\cite{EGbook, HMSW, MJensen}).
This interest derived from numerical analysis motivated the introduction of abstract theory, which was made in \cite{EGC}, for real Hilbert spaces, while the results in the complex setting are rigorously written in \cite{ABCE}.
More precisely,
\begin{definition}\label{def:abstractFO}
A (densely defined) linear operator $T$ on a complex Hilbert space $\lH$ 
is called an \emph{abstract Friedrichs operator} if it admits another
(densely defined) linear operator $\widetilde{T}$ on $\lH$ with the following properties:
\begin{itemize}
 \item[(T1)] $T$ and $\widetilde{T}$ have a common domain $\lD$, 
 which is dense in $\lH$, satisfying
 \[
 \scp{T\phi}\psi \;=\; \scp\phi{\widetilde T\psi} \;, \qquad \phi,\psi\in\mathcal{D} \,;
 \]
 \item[(T2)] there is a constant $c>0$ for which
 \[
 \|(T+\widetilde{T})\phi\| \;\leqslant\; c\|\phi\| \;, \qquad \phi\in\mathcal{D} \,;
 \]
 \item[(T3)] there exists a constant $\mu_0>0$ such that
 \[
 \scp{(T+\widetilde{T})\phi}\phi \;\geqslant\; 2\mu_0 \|\phi\|^2 \;, \qquad \phi\in\mathcal{D} \,.
 \]
\end{itemize}
The pair $(T,\widetilde{T})$ is referred to as a \emph{joint pair of abstract Friedrichs operators}
(the definition is indeed symmetric in $T$ and $\widetilde{T}$).
\end{definition}

The abstract theory encompasses the classical theory (i.~e.~a certain realisation of the operator $L$ from Definition \ref{def:CFO} can be understood as an abstract Friedrichs operator; see Example \ref{ex:cfo} below) and provides a well-posedness result. Here we break down the development of the abstract approach so far.
\begin{itemize}
    \item[(i)] The authors in \cite{EGC} proved the well-posedness result under the assumption of a set of (V)-boundary conditions, calling the \emph{cone formalism}, which is analogous to (FV)-boundary condition in the classical setting. Abstract formulation of (FM) 
    and (FX) boundary conditions is studied in the same reference, as (M) and (X) boundary conditions. However, the equivalence among these boundary conditions was not achieved in full generality.
    \item[(ii)] Connection of the abstract theory with the \emph{Kre\u\i n space} theory is studied in \cite{ABcpde}. The equivalence among the three types of boundary conditions in the abstract setting is proved in full generality.
    \item[(iii)] The non-stationary theory (in terms of the semigroup theory) suitable for evolution equations is developed in \cite{BEmjom}.
    \item[(iv)] In \cite{AEM-2017}, the authors formulated the abstract theory using the operator theoretic approach. Moreover, they proved the existence of a set of (V)-boundary conditions using the \emph{Kre\u\i n space} theory, for any pair of abstract Friedrichs operators. 
    \item[(v)] A simple and equivalent characterisation of the abstract Friedrichs operators and the developement of the extension theory in the spirit of the \emph{von Neumann theory} can be found in \cite{ES23} (see also \cite{ES22}).
    \item[(vi)] The approach of studying classical theory using abstract Friedrichs operators initiated a number of new investigations in various directions. For example, studies of different representations of boundary conditions and the relation with the classical theory \cite{ABjde, ABVisrn, AEM-2017, BH21, BPhd, ES22}, applications to diverse (initial-)boun\-dary value problems of elliptic, hyperbolic, and parabolic type \cite{ABVjmaa, BEmjom, BEW23, BVcpaa, EM19, EGsemel, MDS}, and the development of different numerical schemes \cite{BDG, BEF, CM21, CHWY23, EGbis, EGter}.
\end{itemize}

This operator theoretic framework provides us a convenient way to study the classical theory using the inner-product space theory (Hilbert and Kre\u\i n spaces). As we continue to develop the theory in this direction, in parallel, we seek to apply it on classical Friedrichs operators and continue to explore the realm of the rich framework. 

We take one more step towards this connection by studying the classical Friedrichs systems on intervals. More precisely, we are interested in one dimensional $(d=1)$ vectorial case. The first step was done in \cite{ES22}, by providing a complete classification of boundary conditions in the one dimensional $(d=1)$ scalar $(r=1)$ case. The problem becomes challenging due to non-smoothness of eigenvectors of the coefficient matrices appearing in the definition of (CFO). We overcome this difficulty by making use of some result related to eigenprojections. 

The paper is organised as follows. In Section \ref{sec:abstractFO}, we recall the theory of abstract Friedrichs operators in the direction of results proved in this paper. In Section \ref{sec:total projections}, we obtain a smoothness result on total projections and define the boundary operator and minimal domain explicitly. In Section \ref{sec:1-d-vec} we prove our main result (Theorem \ref{thm:codim}) which is a direct relation on the dimension of kernels and the rank of coefficient matrix appearing in (CFO) at end-points of the interval. In Section \ref{sec:examples}, we illustrate our results on second order linear ordinary differential equations and on a system of first order. In Appendix (Section \ref{sec:kato}) the necessary ingredients related to the smoothness of eigenvalues and total projections along with a codimension lemma have been presented.

\noindent\textbf{Notation.} Most of our notations are standard, 
let us only emphasise the following. 
For the sake of generality, in the paper we work on 
complex vector spaces. Thus, by $\lH$ we denote 
a complex Hilbert space with scalar product 
$\scp\cdot\cdot$, which we take to be 
linear in the first and anti-linear in the second entry.
The corresponding norm is given by
$\|\cdot\|:=\sqrt{\scp\cdot\cdot}$.
For $\lH=\Cr$ we shall often use an alternative notation:
$\scp{\mathsf{x}}{\mathsf{y}} = \mathsf{x}\cdot\mathsf{y}$,
$\mathsf{x}, \mathsf{y}\in\Cr$. 
We shall denote any norm on a finite-dimensional space by $|\cdot|$, in order to distinguish it from the norm of the Hilbert space.
The topological (anti)dual 
$\lH'$ will be identified with $\lH$ by means of the usual duality 
(the Riesz representation theorem).
For any Banach space $\lX$ by $\dup{\lX'}\cdot\cdot \lX$ 
we denote the corresponding dual product between 
$\lX$ and its (anti)dual $\lX'$.

For a densely defined linear operator $A:\lH\to \lH$ we denote by 
$\dom A$, $\ker A$, $\ran A$, $\overline{A}$, $A^\ast$ 
its \emph{domain}, \emph{kernel}, \emph{range} (or \emph{image}), 
\emph{closure} (if it exists), and \emph{adjoint}, respectively.
In addition we use $\rank A= \dim \ran A$.
For $S\subseteq \lH$, the \emph{restriction} of $A$ to $S$ is denoted 
by $A|_S$.
For two linear operators $A, B$ in $\lH$ by $A\subseteq B$
we mean that $\dom A\subseteq \dom B$ and $B|_{\dom A}=A$.
If $A=A^\ast$, 
then $A$ is said to be \emph{self-adjoint}, while the infimum of its spectrum is called the \emph{bottom}. 
The \emph{identity} operator is denoted by $\mathbbm{1}$.
For a \emph{direct} sum between two vector spaces we use the symbol $\dotplus$.

For any complex number $z\in\C$ we denote by $\Re z$ and $\Im z$
the real and the imaginary part of $z$, respectively.

\section{Introduction}\label{sec:abstractFO}
Let us briefly recall the properties of Friedrichs operators and prepare necessary ingredients to obtain the results of this paper. We briefly recall the theory in this section, while for more details, we refer to \cite{EGC, ABcpde, AEM-2017}. Let $(T, \widetilde T)$ be a joint pair of abstract Friedrichs operators on a Hilbert space $\lH$ as in Definition \ref{def:abstractFO}. Let $T_0:=\overline T$ and $\widetilde{T}_0:=\overline{\widetilde{T}}$ be the closures with $\dom T_0 = \dom \widetilde{T}_0 =: \lW_0$. Then the pair $(T_0,\widetilde{T}_0)$ is also a joint pair of abstract Friedrichs operators on $\lH$. Moreover, we define $T_1:= \widetilde{T}^*$, $\widetilde{T}_1: =T^*$ with $\dom T_1=\dom\widetilde{T}_1=:\lW$ (these follow easily from assumption (T1)--(T2), see e.g.~\cite{AEM-2017, EGC}). We summarise some properties in the following theorem (an overview of the theory can be found in \cite[Theorem 2.2]{ES23}). 

\begin{theorem}\label{thm:abstractFO}
    Let $(T_0, \widetilde{T}_0)$ be a joint pair of closed abstract Friedrichs operators on a Hilbert space $\lH$. Then the following results hold.
\begin{enumerate}
\item[\rm{(i)}] The graph norms $\|\cdot\|_{T_1}:= \|\cdot\|+\|T_1\cdot\|$ and $\|\cdot\|_{\widetilde{T}_1}:= \|\cdot\|+\|\widetilde{T}_1\cdot\|$ are equivalent, the graph space $(\lW, \|\cdot\|_{T_1})$ is a Hilbert space and $\lW_0$ is a closed subspace of $\lW$.
    
\item[\rm{(ii)}] The operator $\overline{(T_0+\widetilde{T}_0)}$ is everywhere defined, bounded, self-adjoint operator with strictly positive bottom and it coincides with $T_1+\widetilde{T}_1$ on $\lW$.

\item[\rm{(iii)}]  The continuous linear map defined by,
\begin{align}
& D : (\lW,\|\cdot\|_{T_1}) \to (\lW,\|\cdot\|_{T_1})' \nonumber \\
    \iscp uv \;:=\; & \dup{\lW'}{Du}{v}\lW \;:=\; \scp{T_1u}{v} 
        - \scp{u}{\widetilde{T}_1v} \;,
        \quad u,v\in\lW \;, \label{eq:D}
    \end{align}
    is called the \emph{boundary operator} associated with the pair $(T,\widetilde{T})$. Moreover, $\ker D= \lW_0$, $\lW_0^{[\perp]}= \lW$ and $\lW^{[\perp]}= \lW_0$. Here, $S^{[\perp]}$ is the orthogonal complement of a subset $S$ of $\lW$ with respect to $\iscp{\cdot}{\cdot}$, defined by
    \begin{align*}
        S^{[\perp]}:= \{u\in \lW: \ \iscp{u}{v}=0, \  v\in S\}.
    \end{align*}
It is closed in $\lW$.
\item[\rm{(iv)}] A subspace $\lV$ of $\lW$ which contains $\lW_0$
is closed in $\lW$ (i.e.~with respect to the graph norm)
if and only if $\lV=\lV^{[\perp][\perp]}$.

\item[\rm{(v)}] $(\lW, \iscp{\cdot}{\cdot})$ is an indefinite inner product space. 

\item[\rm{(vi)}] The following decomposition holds
\begin{equation}\label{eq:decomposition}
    \lW\;=\;\lW_0\;\dotplus\; \ker T_1\;\dotplus\;\ker\widetilde{T}_1\;.
\end{equation}
The non-orthogonal projectors $p_0,\ p_{\rm{k}}$ and $p_{\rm{\tilde k}}$ (corresponding to \eqref{eq:decomposition}) over $\lW_0,\ \ker T_1$ and $\ker \widetilde{T}_1$, respectively, are continuous linear operators with respect to the graph norm.
\item[\rm{(vii)}] The subspaces $\lW_0, \ker T_1$ and $\ker \widetilde{T}_1$ are mutually $\iscp{\cdot}{\cdot}-$orthogonal.
\item[\rm{(viii)}] The spaces $(\ker T_1,-\iscp{\cdot}{\cdot})$ and $(\ker \widetilde{T}_1,\iscp{\cdot}{\cdot})$ are inner product spaces and the quotient space $\lW/\lW_0$ is isomorphic to $\ker T_1\dot{+}\ker\widetilde{T}_1$.
\end{enumerate}
    
\end{theorem}
The statements (i)--(iii) can be found in \cite{AEM-2017, EGC} (see also \cite{ABCE} for the complex setting). For (iv), (v) and a further structure 
of indefinite inner product space $(\lW, \iscp{\cdot}{\cdot})$ we refer to 
\cite{ABcpde}. The results (vi), (vii) and (viii) can be found in \cite{ES22}.

\begin{definition}[(V)-boundary conditions]\label{dfn:V1-V2}
    Let $(T_0,\widetilde{T}_0)$ be a joint pair of closed abstract Friedrichs operators on a Hilbert space $\lH$. A pair of subspaces $(\lV, \widetilde \lV)$ of the graph space $\lW$ is said to allow (V)-\emph{boundary conditions} related to $(T,\widetilde T)$ if the following conditions are satisfied:
    \begin{itemize}
        \item[(V1)] The boundary operator has \emph{opposite signs} on these subspaces, i.e.
        \begin{align*}
            & (\forall u\in \lV)\qquad \iscp{u}{u}\;\geq\; 0\,,\\
            & (\forall v\in \widetilde \lV) \qquad \iscp{v}{v}\;\leq\; 0\;.
        \end{align*}
        
        \item[(V2)] The subspaces $\lV, \widetilde \lV$ are $\iscp{\cdot}{\cdot}$-orthogonal complements to each other, i.e.
        \begin{align*}
            \lV \;=\; \widetilde \lV^{[\perp]} \quad \mathrm{and} \quad \widetilde \lV \;=\; \lV^{[\perp]}\;.
        \end{align*}
    \end{itemize}
\end{definition}

\begin{theorem}\label{thm:well-posedness}
     Let $(T_0, \widetilde{T}_0)$ be a joint pair of closed abstract Friedrichs operators on a Hilbert space $\lH$.
\begin{itemize} 
    \item[\rm{(i)}] If a pair $(\lV,\widetilde{\lV})$ satisfies \emph{(V)}-boundary conditions, then  
    \begin{align*}
        T_1|_{\lV}:\lV\to \lH \quad \mathrm{and}\quad \widetilde{T}_1|_{\widetilde{\lV}}:\widetilde{\lV}\to \lH
    \end{align*}
    are bijective.
    \item[\rm{(ii)}] If both $\ker T_1\neq \{0\}$ and 
    $\ker\widetilde{T}_1\neq \{0\}$, then the pair
    $(T,\widetilde{T})$ admits uncountably many mutually adjoint
    pairs of bijective realisations relative to $(T,\widetilde{T})$.
    On the other hand, if either $\ker T_1=\{0\}$ or
    $\ker \widetilde{T}_1=\{0\}$, then there is exactly one mutually 
    adjoint pair of bijective realisations relative to
    $(T,\widetilde{T})$. Such a pair is precisely
    $(T_1,\widetilde T)$ when $\ker T_1=\{0\}$, 
    and $(T,\widetilde{T}_1)$ when $\ker \widetilde{T}_1=\{0\}$.
    \item[\rm{(iii)}] Let $\lV $ be a closed subspace  of  $\lW$ such that $\lW_0\subseteq \lV$. Then $T_1|_{\lV}$ is bijective
 if and only if\/ $\lV\dot{+}\ker T_1=\lW$. 
 \item[\rm{(iv)}] $(\lW_0\dotplus \ker\widetilde{T}_1,\lW_0\dotplus\ker T_1)$ satisfies \emph{(V)}-boundary conditions.
 \item[\rm{(v)}]
 There exists $\lV\subseteq\lW$ such that $(\lV,\lV)$ satisfies \emph{(V)}-boundary conditions if and only if $\ker T_1$ and $\ker \widetilde T_1$ are isomorphic.
\end{itemize}
\end{theorem}
Statement (i) is proved in \cite{EGC} for real vector spaces and for complex vector spaces we refer to \cite{ABCE}. For (ii), we refer to \cite{AEM-2017}. Results (iii) and (iv) can be found in \cite{ES22}.
Finally, the characterisation given in (v) is obtained in \cite{ES23}.

Before moving to the main results of this paper, let us briefly present how the abstract setting encompasses the classical setting, while for details we refer to \cite[Subsection 5.1]{EGC}. A special attention is made on the one dimensional $(d=1)$ vectorial $(r\in \N)$ setting, which is in our focus for the rest of the paper.

\begin{example}[Classical is abstract]\label{ex:cfo}
Let $d,r\in\N$ and $\Omega\subseteq\Rd$ be an open and bounded set with 
Lipschitz boundary $\Gamma$. We consider the restriction of operator $L$ given by (CFO) to $C^\infty_c(\Omega;\Cr)$ 
and denote it by $T$, i.e.
$$
T\vu =  \sum_{k=1}^d \partial_k(\mA_k \vu) + \mB \vu \,,
    \quad \vu\in C^\infty_c(\Omega;\Cr)
$$
(here the derivatives can be understood in the classical sense 
as derivatives of smooth functions are equal to their distributional
derivatives).
Since $\mB \in L^\infty(\Omega; \Mat{r})$ and 
$\mA_k \in W^{1,\infty}(\Omega;\Mat{r})$ (for any $k$),
it is obvious that $T:C^\infty_c(\Omega;\Cr)\to L^2(\Omega;\Cr)$.

For the second operator we take 
$\widetilde T:C^\infty_c(\Omega;\Cr)\to L^2(\Omega;\Cr)$ given by
$$
\widetilde T\vu =  -\sum_{k=1}^d \partial_k(\mA_k \vu) + 
    \Bigl(\mB^* + \sum_{k=1}^d \partial_k\mA_k \Bigr)\vu \,,
    \quad \vu\in C^\infty_c(\Omega;\Cr) \,.
$$
Then one can easily see that $(T,\widetilde T)$ is a joint pair of
abstract Friedrichs operators, where $\lH=L^2(\Omega;\Cr)$
and $\lD=C^\infty_c(\Omega;\Cr)$.
Indeed, (T1) is obtained by integration by parts and using (F1), 
the boundedness of coefficients implies (T2), while (T3) follows from (F2)
(a more general case where $\lH$ is taken 
to be a closed subspace of $L^2(\Omega;\Cr)$ can be found in 
\cite[Example 2]{ABCE}).

The domain of adjoint operators $T_1=\widetilde T^*$
and $\widetilde T = T^*$ (the graph space) reads
\begin{equation*}
\begin{aligned}
\lW &= \Bigl\{\vu\in L^2(\Omega;\Cr) :  
    \sum_{k=1}^d \partial_k(\mA_k \vu) + \mB \vu\in L^2(\Omega;\Cr)\Bigr\} \\
&= \Bigl\{\vu\in L^2(\Omega;\Cr) :  
    \sum_{k=1}^d \partial_k(\mA_k \vu) \in L^2(\Omega;\Cr)\Bigr\} \,.
\end{aligned}
\end{equation*}
The action of $T_1$ and $\widetilde{T}_1$ is (formally) the same as the action
of $T$ and $\widetilde T$, respectively (we have just that the classical 
derivatives are replaced by the distributional ones). 
It is known that $C^\infty_c(\Rd;\Cr)$ is dense in $\lW$ 
\cite[Theorem 4]{ABmc} (cf.~\cite[Chapter 1]{MJensen})
and that the boundary operator, for $\vu,\vv\in C^\infty_c(\Rd;\Cr)$,
is given by
\begin{equation*}
\iscp{\vu}{\vv} = \int_\Gamma \mA_\mnu (\mx)\vu|_\Gamma(\mx)
    \cdot\vv|_\Gamma(\mx)\,dS(\mx) \;,
\end{equation*}
where $\mA_\mnu:=\sum_{k=1}^d\nu_k\mA_k$ and 
$\mnu=(\nu_1,\nu_2,\dots,\nu_d)\in L^\infty(\Gamma;\Rd)$ is the unit
outward normal on $\Gamma$. By the definition, we have that the domain of closures
$T_0=\overline{T}$ and $\widetilde{T}_0=\overline{\widetilde{T}}$
is given by $\lW_0=\cl_\lW C^\infty_c(\Omega;\Cr)$, 
while by Theorem \ref{thm:abstractFO}(iii) and the identity above we have
\begin{equation*}
\begin{aligned}
\lW_0\cap C^\infty_c(\Rd;\Cr) = \Bigl\{\vu\in C^\infty_c(\Rd;\Cr) :  
    (\forall \vv &\in C^\infty_c(\Rd;\Cr)) \\
&\int_\Gamma \mA_\mnu (\mx)\vu|_\Gamma(\mx)
    \cdot\vv|_\Gamma(\mx)\,dS(\mx) = 0 \Bigr\} \;.
\end{aligned}
\end{equation*}
A more specific characterisation involving the trace operator
on the graph space can be found in \cite{ABmc, MJensen}.

In the one-dimensional case ($d=1$) for $\Omega=(a,b)$, $a<b$, the graph space simplifies to
\begin{equation}
\lW\;=\; \bigl\{ \vu\in\lH : (\mA \vu)'\in\lH\bigr\} \;,
\end{equation}
while the graph norm is (equivalent to)
\begin{align}
    \|\cdot\|_{T_1}\;=\; \|\cdot\|+ \|(\mA \;\cdot\,)'\|\;.
\end{align}
The boundary operator $D$ is given by
\begin{equation}\label{eq:D-scalar1d}
\iscp{\vu}{\vv}= \bigl(\mA \vu\cdot  \vv\bigr)(b) -
    \bigl(\mA \vu\cdot  \vv\bigr)(a) \;, \quad \vu,\vv\in C^\infty_c((a,b);\C^r)\;,
\end{equation}
and the minimal domain is described as
\begin{equation}\label{eq:W0}
    \lW_0\;=\;\{\vu\in \lW : (\mA \vu)(a)=(\mA\vu)(b)=0 \}\;.
\end{equation}

\end{example}

\section{total projections}\label{sec:total projections}
    
The study and classification of all the boundary conditions has been done for one-dimensional ($d=1$) scalar ($r=1$) case in full generality \cite{ES22}.
A generalisation to the vectorial case (still in one dimension), becomes particularly challenging, due to the non-smoothness of eigenvectors of smooth matrices. We overcome this difficulty in this section and realise the connection between the kernels and ranks of the coefficient matrix  $\mA(x)$ at end points $a$ and $b$ of the interval $(a,b)$.
Also, we relate this information to the existence of bijective realisations $(T_1|_\lV,\widetilde{T}_1|_{\widetilde\lV})$ such that $\lV=\widetilde \lV$. 

In this section and in the rest of the manuscript we study 
(CFO) in the one-dimensional ($d=1$) vectorial ($r\in \mathbb{N}$) case. 
For the domain we take an open interval $\Omega=(a,b)$,
$a<b$. Then $\lD=C^\infty_c((a,b);\C^r)$ and $\lH=L^2((a,b);\C^r)$.
We adjust the notation of $T, \widetilde T:\lD\to\lH$
given in Example \ref{ex:cfo} in the following way:
\begin{equation}\label{eq:TTtilda}
T\vu := (\mA \vu )'+\mB \vu
	\qquad \hbox{and} \qquad
	\widetilde{T}\vu:=-(\mA \vu)'
	+(\mB^*+\mA ')\vu \;,
\end{equation}
where $\mA\in W^{1,\infty}((a,b);\Mat{r})$,
$\mB\in L^\infty((a,b);\C^r)$ and for some $\mu_0>0$
we have $\mB^*+\mB  +\mA '\geq 2\mu_0\mathbbm{1}>0$
($\mathbbm{1}$ is the identity matrix
and $'$ the derivative). 

It is commented in Example \ref{ex:cfo} that 
$(T, \widetilde T)$ 
is a joint pair of abstract Friedrichs operators.
Moreover, the graph space is given by 
\begin{equation*}
\lW \;=\; \bigl\{ \vu\in\lH : (\mA \vu)'\in\lH\bigr\} \;,
\end{equation*}
while the graph norm is equivalent to 
$\|\vu\|_{\mathcal{W}}:=\|\vu\|+\|(\mA \vu)'\|$
($\|\cdot\|$ stands, as usual, for the norm on $\lH$
induced by the standard inner product, i.e.~the $L^2$ norm
on $(a,b)$).
In fact, $\vu\in\lH$ belongs to $\lW$ if and only if 
$\mA \vu\in H^1((a,b);\C^r)$. Thus, by the standard Sobolev 
embedding theorem (see e.g.~\cite[Theorem 8.2]{Brezis})
for any $\vu\in\lW$ we have $\mA \vu\in C([a,b];\C^r)$.
This in particular implies that for any $\vu\in\lW$ and
$x\in[a,b]$
evaluation $(\mA \vu)(x)$ is well defined.
On the other hand, $\mA(x)\vu(x)$ is not necessarily
meaningful as $\vu$ itself is not necessarily continuous.
However, we can define this evaluation map as a consequence of the following lemma.

Let $\mP_{\lambda}(x_0)$ denotes the total projection corresponding to the eigenvalue $\lambda = \lambda(x_0)$ of $\mA(x_0)$. As mentioned in Appendix, for any 
$x_0\in [a,b]$ there exists $\varepsilon =:\varepsilon(x_0) >0$, such that the operator $\mP_{\lambda}(x_0)$ is Lipschitz continuous in the interval $[a,b]\cap [x_0-\varepsilon, x_0+\varepsilon] =: I_{\lambda,x_0}$. Which means $\mP_{\lambda}\in W^{1,\infty}(I_{\lambda,x_0};\C^{r\times r})$.

Also, eigenvalues $\lambda(x)$ of the matrix $\mA(x)$ are in $W^{1,\infty}((a,b),\R)$.

\begin{lemma}\label{lm:smth}
    Let $\lambda(x)$ be an eigenvalue of $\mA(x)$ and $x_0\in [a,b]$, such that $\lambda:=\lambda(x_0)\neq 0$. Then there exists $\varepsilon >0$ such that for any $\vu\in \lW$, 
    \begin{align*}
        (\mP_{\lambda}\vu)|_{I_{\lambda,x_0}}\in H^1(I_{\lambda,x_0};\C^r)\,,
    \end{align*}
    where $I_{\lambda,x_0}=[a,b]\cap [x_0-\varepsilon, x_0+\varepsilon]$.
\end{lemma} 
\begin{proof}
Let $\varepsilon>0$ such that $\mP_{\lambda}\in W^{1,\infty}(I_{\lambda,x_0};\C^{r\times r})$ and for any $x\in I_{\lambda,x_0}$, eigenvalues of $\mA(x)$ forming the $\lambda-$group (see Appendix) are $\lambda/2$ close to $\lambda$ (for simplicity, we assumed $\lambda>0$). Hence, for any $ x\in I_{\lambda,x_0}$ and $ i\in I(x)$, it holds  
\begin{align*}
    |\lambda_i(x)-\lambda|\leq \lambda/2 \iff \lambda_i(x)\in [\lambda/2,3\lambda/2]\,,
\end{align*}
where by $I(x)$, $x\in I_{\lambda,x_0}$, we denote the index set of eigenvalues forming the $\lambda-$group.
Let $x\in I_{\lambda,x_0}$. For an arbitrary $\vnu\in \C^r$, we have
\begin{align*}
    |\mP_{\lambda}(x)\vnu|^2
    &= (\mP_{\lambda}(x)\vnu)\cdot (\mP_{\lambda}(x)\vnu)\\
    &=\bigl(\sum_{i\in I(x)} \mP_{i}(x)\vnu\bigr)\cdot \bigl(\sum_{j\in I(x)} \mP_{j}(x)\vnu\bigr) \\   
    &= \sum_{i,j\in I(x)}\bigl( \mP_j^*(x)\mP_{i}(x)\vnu\bigr)\cdot \vnu\;.
\end{align*}
Here, $\mP_j(x)$ is the eigenprojection corresponding to the eigenvalue $\lambda_j(x)$. Applying $\mP^*_j(x)=\mP_j(x)$ and $\mP_i(x)\mP_j(x)=\delta_{i,j}\mP_i(x)$ (see Appendix), we get
\begin{equation}\label{eq:eigen-proj}
     |\mP_{\lambda}(x)\vnu|^2=\sum_{i\in I(x)} | \mP_{i}(x)\vnu|^2\;.
\end{equation}
Since, $\lambda/2\leq \lambda_i(x)$, and $\mP_i(x)$ and $\mA(x)$ commute (note that $\mA(x)$ is Hermitian), $i\in I(x)$, we have
\begin{align*}
     |\mP_{i}(x)\vnu|^2 \leq \frac{4}{\lambda^2}|\lambda_i(x)\mP_{i}(x)\vnu|^2
     = \frac{4}{\lambda^2}|\mA(x)\mP_{i}(x)\vnu|^2
     =\frac{4}{\lambda^2}|\mP_{i}(x)\mA(x)\vnu|^2\;.
\end{align*}
Replacing $\vnu$ by $\mA(x)\vnu$ in \eqref{eq:eigen-proj}, we obtain 
\begin{align*}
    \sum_{i\in I(x)} |\mP_{i}(x)\mA(x)\vnu|^2= |\mP_{\lambda}(x)\mA(x)\vnu|^2\leq |\mA(x)\nu|^2\,,
\end{align*}
where we used that $\mP_{\lambda}(x)$ is an orthogonal projection. Thus, for any $x\in I_{\lambda,x_0}$ and $\nu\in\C^r$ we have
\begin{align*}
    |\mP_{\lambda}(x)\vnu|^2\leq \frac{4}{\lambda^2}|\mA(x)\vnu|^2\;.
\end{align*}
Now, let us take $\vu \in C_c^{\infty}(\R; \C^r)$. Then, the above inequality holds for $\vu(x)$ as well as $\vu'(x)$ in place of $\vnu$. Integrating over $I_{\lambda,x_0}$ and taking the square root, we get
\begin{align*}
    \|\mP_{\lambda}\vu'\|_{L^2(I_{\lambda,x_0})}\leq \frac{2}{\lambda}\|\mA\vu'\|_{L^2(I_{\lambda,x_0})}\leq \frac{2(1+\|A'\|_{L^{\infty}(I_{\lambda,x_0})})}{\lambda}\|\vu\|_{\lW}\;.
\end{align*}
Hence,
\begin{align*}
    \|(\mP_{\lambda}\vu)'\|_{L^2(I_{\lambda,x_0})}\leq C\|\vu\|_{\lW}\;.
\end{align*}
Where, $C=2\max\{\|\mP_{\lambda}'\|_{L^{\infty}(I_{\lambda,x_0})}, \frac{2(1+\|A'\|_{L^{\infty}(I_{\lambda,x_0})})}{\lambda}\}$.
Since, $C_c^{\infty}(\R;\C^r)$ is dense in $\lW$, we conclude that for any $\vu\in \lW$, we have
\begin{align*}
    (\mP_{\lambda}\vu)|_{I_{\lambda,x_0}}\in H^1(I_{\lambda,x_0};\C^r)\;.
\end{align*}

\end{proof}

\begin{remark}\label{rem:h1loc}
 Using Sobolev embedding, we get $(\mP_{\lambda}\vu)|_{I_{\lambda,x_0}}\in C(I_{\lambda,x_0};\C^r)$ and so pointwise evaluation is well-defined. Of course, if for $x\in I_{\lambda,x_0}$ the evaluation $\vu(x)$ is well-defined, then 
 $$(\mP_{\lambda}\vu)(x)=\mP_{\lambda}(x)\vu(x)\,,$$
 where we used that $\mP_{\lambda}$ is continuous (see Appendix).

 Moreover, it is easy to see that for any $\varphi\in W^{1,\infty}((a,b))$ and $\vu\in\lW$ we have $\varphi \vu\in\lW$ and $(\mP_\lambda(\varphi\vu))(x)=\varphi(x)(\mP_\lambda\vu)(x)$.
    
\end{remark}

The boundary operator on $\lW$ and the minimal space $\lW_0$ can be defined more explicitly.
\begin{lemma}\label{lm:DW0}
 Let $\sigma(\mA(x))$ denotes the spectrum of matrix $\mA(x)$.
     \begin{enumerate}
         \item[\rm{(i)}] For any $\vu,\vv\in \lW$, the boundary operator can be characterised as
         \begin{align*}
             \iscp{\vu}{\vv}= \sum_{\lambda\in \sigma(\mA(b))\setminus \{0\}}\lambda(\mP_{\lambda}\vu)(b)\cdot (\mP_{\lambda}\vv)(b) \, -\!\!\sum_{\lambda\in \sigma(\mA(a))\setminus \{0\}}\lambda(\mP_{\lambda}\vu)(a)\cdot (\mP_{\lambda}\vv)(a) \;.
         \end{align*}
         \item[\rm{(ii)}] The domain of the closures $T_0$ and $\widetilde{T}_0$ is characterised as
         \begin{align*}
             \lW_0=\bigl\{\vu\in \lW:(\forall x\in \{a,b\})\; (\forall \lambda\in \sigma(\mA(x))\setminus \{0\})\; (\mP_{\lambda}\vu)(x)=0 \bigr\}\;.
         \end{align*}
     \end{enumerate}
\end{lemma}

\begin{proof}
    \begin{enumerate}
        \item[(i)] Let $\vu,\vv\in C_c^{\infty}(\R,\C^r)$, then (see \eqref{eq:D-scalar1d})
        \begin{align*}
            \iscp{\vu}{\vv}&=\mA(b)\vu(b)\cdot \vv(b)-\mA(a)\vu(a)\cdot \vv(a)\\
            &= \sum_{\lambda\in \sigma(\mA(b))\setminus \{0\}}\lambda \,\mP_{\lambda}(b)\vu(b)\cdot \vv(b) -\sum_{\lambda\in \sigma(\mA(a))\setminus \{0\}}\lambda \,\mP_{\lambda}(a)\vu(a)\cdot \vv(a)\;.
        \end{align*}
        Using $\mP_{\lambda}(x)=\mP_{\lambda}(x)^2$ and $\mP^*_{\lambda}(x)=\mP_{\lambda}(x)$ (see also Remark \ref{rem:h1loc}), we get
        \begin{align*}
            \iscp{\vu}{\vv}= \sum_{\lambda\in \sigma(\mA(b))\setminus \{0\}}\lambda(\mP_{\lambda}\vu)(b)\cdot (\mP_{\lambda}\vv)(b) -\sum_{\lambda\in \sigma(\mA(a))\setminus \{0\}}\lambda(\mP_{\lambda}\vu)(a)\cdot (\mP_{\lambda}\vv)(a) \;.
        \end{align*}
        Now, let $\vu,\vv\in \lW$. Then there exist sequences $\vu_n,\vv_n\in C_c^{\infty}(\R,\C^r)$ such that $\vu_n,\vv_n$ converge in $\lW$ to $\vu,\vv$, respectively. 
        By Lemma \ref{lm:smth} (see also Remark \ref{rem:h1loc}),
        $(\mP_{\lambda}\vu_n)(a)\to (\mP_{\lambda}\vu)(a)$ and $(\mP_{\lambda}\vu_n)(b)\to (\mP_{\lambda}\vu)(b)$. Hence,
        for any $\vu,\vv\in \lW$, 
        \begin{align*}
            \iscp{\vu}{\vv}= \sum_{\lambda\in \sigma(\mA(b))\setminus \{0\}}\lambda(\mP_{\lambda}\vu)(b)\cdot (\mP_{\lambda}\vv)(b) -\sum_{\lambda\in \sigma(\mA(a))\setminus \{0\}}\lambda(\mP_{\lambda}\vu)(a)\cdot (\mP_{\lambda}\vv)(a) \;.
        \end{align*}
        \item[(ii) ] The form of $\lW_0$ is described in \eqref{eq:W0}. Let $\vu \in  C_c^{\infty}((a,b);\C^r)$, then
        \begin{align*}
            \mA(a)\vu(a)=\sum_{\lambda\in \sigma(\mA(a))\setminus \{0\}}\lambda\mP_{\lambda}(a)\vu(a)=0\;.
        \end{align*}
Since $\mA(a)$ is Hermitian, the eigenprojections are orthogonal to each other. Which implies that $\{\mP_{\lambda}(a)\vu(a):\lambda \in \sigma(\mA(a))\setminus \{0\} \} $ is an orthogonal set. 
Thus, we have
        \begin{align*}
            \mA(a)\vu(a)=0\iff (\forall \lambda\in \sigma(\mA(a))\setminus \{0\})\quad (\mP_{\lambda}\vu)(a)= \mP_{\lambda}(a) \vu(a)=0\;.
        \end{align*}
        
Now, let $\vu \in \lW_0$ and take the sequence $\vu_n \in  C_c^{\infty}((a,b),\C^r)$ such that $\vu_n $ converges to $\vu $ in $\lW$. By Lemma \ref{lm:smth}, $(\mP_{\lambda}\vu_n)(a)\rightarrow (\mP_{\lambda}\vu)(a)$. Hence, we have
\begin{align*}
    (\forall \lambda\in \sigma(\mA(a))\setminus \{0\})\quad (\mP_{\lambda}\vu)(a)=0\;.
\end{align*}
Similarly, 
\begin{align*}
    (\forall \lambda\in \sigma(\mA(b))\setminus \{0\}) \quad (\mP_{\lambda}\vu)(b)=0\,,
\end{align*}
which completes the first inclusion.
For the converse, assume that for some $\vu\in \lW$, 
\begin{equation*}
    (\forall x\in \{a,b\}, \forall \lambda\in \sigma(\mA(x))\setminus \{0\})\quad (\mP_{\lambda}\vu)(x)=0\;.
\end{equation*}
Here, the evaluations are well defined by Remark \ref{rem:h1loc}. Due to the density of $C_c^{\infty}(\R,\C^r)$ in $\lW$, there exist $\vu_n \in C_c^{\infty}(\R,\C^r)$ such that $\vu_n\rightarrow \vu$ in $\lW$. For $x\in \{a,b\}$ we have
\begin{align*}
    (\mA\vu_n)(x)=\mA(x)\vu_n(x)=\!\!\sum_{\lambda\in\sigma(\mA(x))\setminus \{0\}}\lambda \mP_{\lambda}(x)\vu_n(x)=\!\! \sum_{\lambda\in\sigma(\mA(x))\setminus \{0\}}\lambda (\mP_{\lambda}\vu_n)(x)\,.
\end{align*}
By the Sobolev embedding theorem (see Lemma \ref{lm:smth} and Remark \ref{rem:h1loc}), by letting $n\to \infty$ we get
\begin{align*}
    (\mA\vu)(x)=\sum_{\lambda\in\sigma(\mA(x))\setminus \{0\}}\lambda (\mP_{\lambda}\vu)(x)=0\;.
\end{align*}
Hence,
\begin{align*}
    (\mA\vu)(a)=(\mA\vu)(b)=0\,,
\end{align*}
implying $\vu\in \lW_0$ by \eqref{eq:W0}.

    \end{enumerate}
\end{proof}

\section{One-dimensional vectorial (CFO)}\label{sec:1-d-vec}

In this section we shall see how to construct a pair of subspaces $(\lV,\widetilde{\lV})$ of $\lW$, i.e.~how to impose suitable boundary conditions, in order to get bijective realisations of \eqref{eq:TTtilda}. 
The strategy is to reduce the problem to the one-dimensional setting and use the approach of \cite{ES22}.
Of course, in the diagonal case it is easy to proceed with this plan (see Example 1 in Section \ref{sec:examples}), while in the general case we shall make use of the results on total projections just developed. 

More precisely, depending on the sign of eigenvalues of $\mA$ at the end-points, we can construct a
pair of subspaces $(\lV,\widetilde{\lV})$ of $\lW$, forming a pair of mutually adjoint bijective realisations of \eqref{eq:TTtilda}. We first define the subspaces 
$\{\lV_{\lambda,a},\widetilde\lV_{\lambda,a}\}_{\lambda\in \sigma(\mA(a))\setminus \{0\}}$ and $\{\lV_{\lambda,b},\widetilde\lV_{\lambda,b}\}_{\lambda\in \sigma(\mA(b))\setminus \{0\}}$ of $\lW$ as follows:

For $\lambda\in \sigma(\mA(a))\setminus \{0\}$
\begin{align*}
    \begin{tabular}{|c|c|c|c|}
  \hline
  Sign of $\lambda$ & $\lV_{\lambda,a}$ & $\widetilde{\lV}_{\lambda,a}$ \\
  \hline   
   $\lambda=0$ & $\lW $ & $\lW$ \\
   \hline
   $\lambda>0$ & $\{\vu\in \lW: (\mP_{\lambda}\vu)(a)=0\}$ & $\lW$\\
   \hline 
   $\lambda<0$ & $\lW$ & $\{\vu\in \lW: (\mP_{\lambda}\vu)(a)=0\}$ \\
  \hline 
\end{tabular}
\end{align*}
and for $\lambda\in \sigma(\mA(b))\setminus \{0\}$
\begin{align*}
    \begin{tabular}{|c|c|c|c|}
  \hline
  Sign of $\lambda$ & $\lV_{\lambda,b}$ & $\widetilde{\lV}_{\lambda,b}$ \\
  \hline   
   $\lambda=0$ & $\lW $ & $\lW$ \\
   \hline
   $\lambda>0$ & $\lW$ & $\{\vu\in \lW: (\mP_{\lambda}\vu)(b)=0\}$\\
   \hline 
   $\lambda<0$ & $\{\vu\in \lW: (\mP_{\lambda}\vu)(b)=0\}$ & $\lW$ \\
  \hline 
\end{tabular}
\end{align*}
Now we shall see that the subspaces
\begin{equation}\label{eq:construction}
    \lV = \lV_a\cap \lV_b  \quad \hbox{and}\quad 
        \widetilde\lV  = \widetilde{\lV}_a\cap \widetilde\lV_b\,,
\end{equation}
satisfy the condition (V), where
\begin{equation}\label{eq:cons-v}
\begin{aligned}
    &\lV_a:= \bigcap_{\lambda\in \sigma(\mA(a))\setminus \{0\}} \lV_{\lambda,a}\,,\quad\lV_b :=\bigcap_{\lambda\in \sigma(\mA(b))\setminus \{0\}}\lV_{\lambda,b}\,, \\
    & \widetilde\lV_a:= \bigcap_{\lambda\in \sigma(\mA(a))\setminus \{0\}} \widetilde\lV_{\lambda,a}\,,\quad \widetilde\lV_b:= \bigcap_{\lambda\in \sigma(\mA(b))\setminus \{0\}} \widetilde\lV_{\lambda,b}\,.
\end{aligned}
\end{equation}

\begin{lemma}\label{lm:cons}
The pair of subspaces $(\lV,\widetilde{\lV})$ of $\lW$, defined as above, satisfies the condition \rm{(V)}.
\end{lemma}
\begin{proof}
By construction \eqref{eq:construction}, 
\begin{align*}
    (\forall \vu\in \lV)\quad \ \sum_{\lambda\in \sigma(\mA(b))\setminus \{0\}} \lambda|(\mP_{\lambda}\vu)(b)|^2-\sum_{\lambda\in \sigma(\mA(a))\setminus \{0\}}\lambda| (\mP_{\lambda}\vu)(a)|^2\geq 0\;.
\end{align*}
Similarly,
\begin{align*}
    (\forall \vu\in \widetilde\lV)\quad\ \sum_{\lambda\in \sigma(\mA(b))\setminus \{0\}} \lambda|(\mP_{\lambda}\vu)(b)|^2-\sum_{\lambda\in \sigma(\mA(a))\setminus \{0\}}\lambda| (\mP_{\lambda}\vu)(a)|^2\leq 0\;.
\end{align*}
Hence, by Lemma \ref{lm:DW0}(i), (V1) condition is satisfied.

Let $\vv\in \widetilde{\lV}$, then for any $\vu \in \lV$, we have
\begin{align*}
    \iscp{\vu}{\vv}= \sum_{\lambda\in \sigma(\mA(b))\setminus \{0\}}\lambda(\mP_{\lambda}\vu)(b)\cdot (\mP_{\lambda}\vv)(b) -\sum_{\lambda\in \sigma(\mA(a))\setminus \{0\}}\lambda(\mP_{\lambda}\vu)(a)\cdot (\mP_{\lambda}\vv)(a) \;.
\end{align*}
For $\lambda\in \sigma(\mA(a))\setminus \{0\} $, 
\begin{equation*}
    \left\{\begin{array}{lcl}
                       (\mP_{\lambda}\vu)(a)=0, & \mathrm{if}\  \lambda >0\\ 
                       (\mP_{\lambda}\vv)(a)=0, & \mathrm{if}\  \lambda <0 \;.
         \end{array}\right.
\end{equation*}
 Similarly, for $\lambda\in \sigma(\mA(b))\setminus \{0\} $, 
\begin{equation*}
    \left\{\begin{array}{lcl}
                       (\mP_{\lambda}\vv)(b)=0, & \mathrm{if}\  \lambda >0\\ 
                       (\mP_{\lambda}\vu)(b)=0, & \mathrm{if}\  \lambda <0 \;.
         \end{array}\right.
\end{equation*}
Which gives $\iscp{\vu}{\vv}=0$.
Thus, $\widetilde{\lV}\subseteq \lV^{[\perp]}$.

For the converse, let $\vv\in \lV^{[\perp]}$, then for any $\vu\in \lV$, 
\begin{align*}
    \iscp{\vu}{\vv}=0\;.
\end{align*}
If all the eigenvalues of $\mA(b)$ are nonpositive and all the eigenvalues of $\mA(a)$ are nonnegative i.e.~ $\lV=\lW_0$, then $\widetilde\lV = \lW$. Hence, the inclusion $\lV^{[\perp]}\subseteq \widetilde \lV$ is trivial. Let us assume that there exists a strictly positive eigenvalue of $\mA(b)$ (the left end-point is treated in an analogous manner).
Let us denote by $\lambda>0$, an arbitrary such eigenvalue. We choose $\vu \in H^1((a,b);\C^r)$ such that $\vu(a)=0, (\mP_{\lambda}\vu)(b)\neq 0$, and for any $\lambda'\in \sigma(\mA(b))\setminus \{\lambda\})$ we have $(\mP_{{\lambda}'} \vu)(b)= 0$ (e.g.~for $\mathsf{e}\in \C^r$ such that $\mP_{\lambda}(b)\mathsf{e}\neq 0$, we can take $\vu(x)=\bigl(\frac{x-a}{x-b}\bigr)\mP_{\lambda}(b)\mathsf{e}$). It is evident that $\vu\in \lV$. By inserting this $\vu$ in the identity above, we get
\begin{align*}
    \lambda(\mP_{\lambda}\vu)(b)(\mP_{\lambda}\vv)(b)=0\implies \vv \in \widetilde{\lV}_{\lambda,b}\;. 
\end{align*}
Since, $ \lambda\in \sigma(\mA(b))\setminus \{0\}$, $\lambda>0$, was arbitrary, we get $\vv \in \bigcap_{\lambda\in \sigma(\mA(b))\setminus \{0\}} \widetilde\lV_{\lambda,b}$.
Similarly, $\vv \in \bigcap_{\lambda\in \sigma(\mA(a))\setminus \{0\}} \widetilde\lV_{\lambda,a} $.
Hence, $\vv\in \widetilde\lV$, concluding the proof.
\end{proof}

An immediate consequence of the previous lemma (see Theorem \ref{thm:well-posedness}(i)) is that the pair $(T_1|_\lV, \widetilde T_1|_{\widetilde\lV})$
is a pair of mutually adjoint bijective realisations. 
Having this information available, we can by means of Theorem \ref{thm:well-posedness}(iii) get some information on the kernels, which is fundamental in describing all bijective realisations (cf.~\cite{AEM-2017, ES22, ES23}). 
This is precisely the content of the main result of the paper, which is given in the following theorem.

\begin{theorem}\label{thm:codim}
 For operators $T_1=\widetilde{T}^*$ and $\widetilde{T}_1=T^*$, where $T$ and $\widetilde T$ are given by \eqref{eq:TTtilda}, we have
\begin{equation}\label{eq:codim}
\dim \ker T_1= n_a^++n_b^- \quad \mathrm{and} \quad \dim \ker \widetilde{T}_1=n_a^-+n_b^+ \,,
\end{equation}
where $n_x^+, n_x^-$ are the number of positive and negative eigenvalues of the matrix $\mA(x)$, respectively.
\end{theorem} 

\begin{proof}
For any closed subsapce $\lV\subseteq \lW$, $\lW_0\subseteq \lV$, such that $T_1|_{\lV}$ is bijective, we get by Theorem \ref{thm:well-posedness}(iii)
\begin{align*}
    \dim\ker T_1= \dim (\lW/ \lV)\;.
\end{align*}
By Lemma \ref{lm:cons} one such $\lV$ is given by \eqref{eq:construction}. Applying Lemma \ref{lm:quo-sum}, given below, we have
\begin{align*}
    \dim\ker T_1 \;=\; \sum_{\lambda\in \sigma(\mA(a))\setminus \{0\}}\dim\bigl(\lW/\lV_{\lambda,a}\bigr)+\sum_{\lambda\in \sigma(\mA(b))\setminus \{0\}}\dim\bigl(\lW/\lV_{\lambda,b}\bigr)\;.
\end{align*}
Hence, it is left to determine $\dim \bigl(\lW/\lV_{\lambda,a}\bigr)$ and $\dim \bigl(\lW/\lV_{\lambda,b}\bigr)$, for any $\lambda$.

From the construction, for any $\lambda \in \sigma (\mA(a))\setminus\{0\}$ we have 
\[\dim(\lW/\lV_{\lambda,a})\;=\;   \left\{
\begin{array}{ll}
      \operatorname{rank} \mP_{\lambda}(a), &  \lambda >0\,,\\
      0, & \lambda <0\;.\\
\end{array} 
\right. \]
Similarly, for any $\lambda \in \sigma (\mA(b))\setminus\{0\}$,
\[\dim(\lW/\lV_{\lambda,b})\;=\;   \left\{
\begin{array}{ll}
      \operatorname{rank} \mP_{\lambda}(b), & \lambda <0\,,\\
      0, & \lambda >0\;.\\
\end{array} 
\right. \]
This implies
\begin{align*}
    \sum_{\lambda \in \sigma (\mA(a))\setminus\{0\}}\dim\bigl(\lW/\lV_{\lambda,a}\bigr)=n_a^+ \quad \hbox{and} \     \sum_{\lambda \in \sigma (\mA(b))\setminus\{0\}}\dim\bigl(\lW/\lV_{\lambda,b}\bigr)=n_b^-\,, \ \
\end{align*}
Therefore,
\begin{align*}
       \dim\ker T_1 =n_a^++n_b^-\;.
\end{align*}

In a completely analogous way, one obtains 
\begin{align*}
    \dim \ker \widetilde{T}_1=n_a^-+n_b^+\;.
\end{align*}



\end{proof}

In the proof of Theorem \ref{thm:codim} we use the following lemma. 

\begin{lemma}\label{lm:quo-sum}
    Let $\lV$ and $\widetilde \lV$ be given by \eqref{eq:construction}. Then,
    \begin{equation*}
    \begin{aligned}
    &\dim\bigl(\lW/\lV\bigr)=\sum_{\lambda\in \sigma(\mA(a))\setminus \{0\}}\dim\bigl(\lW/\lV_{\lambda,a}\bigr)+\sum_{\lambda\in \sigma(\mA(b))\setminus \{0\}}\dim\bigl(\lW/\lV_{\lambda,b}\bigr)\,,\\  \ \mathrm{and} \quad &\dim\bigl(\lW/\widetilde\lV\bigr)=\sum_{\lambda\in \sigma(\mA(a))\setminus \{0\}}\dim\bigl(\lW/\widetilde\lV_{\lambda,a}\bigr)+\sum_{\lambda\in \sigma(\mA(b))\setminus \{0\}}\dim\bigl(\lW/\widetilde\lV_{\lambda,b}\bigr)\;.
    \end{aligned}
    \end{equation*}
\end{lemma}
\begin{proof}
We shall prove the first equality only, as the second one is completely analogous to it. Let us first show that $\lW = \lV_a+\lV_b$, where $\lV_a$ and $\lV_b$ are given by \eqref{eq:cons-v}.
The subspaces $\lV_a$ and $\lV_b$ can be rewritten as,
\begin{align*}
    \lV_a & = \{\vu \in \lW: (\forall  \lambda \in \sigma(\mA(a)), \lambda >0)\ (\mP_{\lambda}\vu)(a)=0 \} \\
   \mathrm{and}\quad \lV_b &= \{\vu \in \lW: (\forall  \lambda \in \sigma(\mA(b)), \lambda >0)\ (\mP_{\lambda}\vu)(b)=0\}\;.
\end{align*}
Let $\varphi$ be a smooth function on $[a,b]$ such that $\phi(a)=1$ and $\phi(b)=0$. For $\vu\in \lW$, we have $\vu=(1-\varphi)\vu+\varphi\vu$. It is clear that $(1-\varphi)\vu, \varphi\vu\in \lW$ and (see Remark \ref{rem:h1loc})
\begin{align*}
   &(\forall \lambda\in \sigma(\mA(b))\land \lambda<0)\quad  (\mP_{\lambda}(\varphi \vu))(b)=\varphi(b) (\mP_\lambda \vu)(b)=0\,, \\ 
   & (\forall \lambda\in \sigma(\mA(a))\land \lambda>0)\quad (\mP_{\lambda}(1-\varphi) \vu)(a)= (1-\varphi(a))(\mP_\lambda\vu)(a) = 0
\end{align*}
(in fact the above holds regardless of the sign of $\lambda$).
This implies that $\varphi\vu\in \lV_b$ and $(1-\varphi)\vu\in \lV_a$. Since $\vu\in \lW$ was arbitrary, $\lW\subseteq \lV_a+\lV_b$. From the construction, $\lV_a,\lV_b\subseteq \lW$, hence $\lW=\lV_a+\lV_b$. Thus, $\dim(\lW/(\lV_a+\lV_b))=0$. Using Lemma \ref{lm:codimension1}, we get \begin{align*}
    \dim(\lW/ \lV)\;=\;\dim(\lW/ \lV_a)\;+\;\dim(\lW/ \lV_b)\;.
\end{align*}
It is enough to prove that 
\begin{align*}
    \dim(\lW/\lV_a)\;=\;\sum_{\lambda\in \sigma(\mA(a))\setminus \{0\}}\dim\bigl(\lW/\lV_{\lambda,a}\bigr)\;.
\end{align*}

Let $\lambda_{i_1},\lambda_{i_2},\dots,\lambda_{i_n}$ $(n\leq r)$  be distinct positive eigenvalues of $\mA(a)$ and $\lV_k:=\lV_{\lambda_{i_k},a}$, $\mP_k:=\mP_{\lambda_{i_k}}$ for $k\in \{1,2,\dots,n\}$. We claim that for each $k\in\{2,3,\dots,n\}$,
\begin{equation}\label{eq:claim}
    \bigcap_{i=1}^{k-1}\lV_i+\lV_k=\lW\;.
\end{equation}

Let $\varepsilon>0$ be such that all $\mP_k$, $k=1,2,\dots,n$, are well-defined and Lipschitz continuous on $[a, a+2\varepsilon]$ (see Appendix). Let $\vartheta$ be a smooth function on $[a,b]$ such that 
\[\vartheta(x)\;=\;   \left\{
\begin{array}{ll}
      1 \,,& x\in [a, a+\varepsilon] \\
      0 \,,& x\in [a+2\varepsilon, b]
\end{array} \,.
\right. \]
For an arbitrary $\vv\in\lW$ we have $(1-\vartheta)\vv\in \lV_{k}$. Thus, it is sufficient to 
study $\vu:=\vartheta\vv$.
Since $\mP_k$ is well-defined on the support of $\vartheta$, we can write
\begin{align*}
    {\vu} = \mP_{k}\vu+(\mathbbm{1}-\mP_{k })\vu \;,
\end{align*}
where both functions on the right-hand side belongs to the graph space $\lW$.
From the following simple series of equalities
\begin{align*}
(\mP_{k}(\mathbbm{1}-\mP_{k })\vu)(a)=(\mP_{k}\vu)(a)-(\mP_{k}\mP_{k}\vu)(a)=(\mP_{k}\vu)(a)- (\mP_{k}\vu)(a)=0\,,
\end{align*}
we get $(\mathbbm{1}-\mP_k)\vu\in \lV_k$. Also, $(\mP_{k}\mP_{k}\vu)(a)=(\mP_{k}\vu)(a)$. If $(\mP_{k}\vu)(a)=0$, then $\vu\in \lV_k$. Let us assume that $(\mP_{k}\vu)(a)\neq 0$. Then, for any $1\leq i<k$, 
\begin{align*}
    (\mP_i\mP_k\vu)(a)=0\,,
\end{align*}
implying $\mP_k\vu\in \bigcap_{i=1}^{k-1}\lV_i$. Since, $\vv\in \lW$ was arbitrary, $\lW\subseteq \bigcap_{i=1}^{k-1}\lV_i+\lV_k$. Which proves the claim \eqref{eq:claim}. Hence,
\begin{align*}
    (\forall k\in \{2,3,\dots,n\})\quad \dim\Bigl(\lW\setminus \bigcap_{i=1}^{k-1}\lV_i+\lV_k\bigr)=0\;.
\end{align*}
Using Lemma \ref{lm:codimension2}, finally we obtain
\begin{align*}
    \dim(\lW/\lV_a)\;=\;\sum_{\lambda\in \sigma(\mA(a))\setminus \{0\}}\dim\bigl(\lW/\lV_{\lambda,a}\bigr)\;,
\end{align*}
which should have been shown.
\end{proof}


Having the result of Theorem \ref{thm:codim} at our hands, we can formulte the following corollary.

\begin{corollary} The codimension of the graph space $\lW$ over the minimal space $\lW_0$ is given by
\begin{equation*} 
\dim\bigl(\lW/{\lW_0}\bigr)= \operatorname{rank} \mA(a) +\operatorname{rank} \mA(b)\;.    
\end{equation*}
\end{corollary}
\begin{proof}
By the decomposition \eqref{eq:decomposition}
    we have
    $$
    \dim(\ker T_1) + \dim(\ker\widetilde{T}_1) = 
       \dim \bigl(\lW/\lW_0\bigr) \,.
    $$
So, by the previous theorem, we have
\begin{align*}
           \dim \bigl(\lW/\lW_0\bigr)= n_a^++n_b^-+n_a^-+n_b^+= \operatorname{rank} \mA(a) +\operatorname{rank} \mA(b) \,.
\end{align*}
\end{proof}
\begin{remark}
    The triplet $(n_x^+,n_x^0,n_x^-)$ is called the inertia of the Hermitian matrix $\mA(x)$ which is relevant in Sylvester's law of inertia, where $n_x^0$ denotes the number of zero eigenvalues of $\mA(x)$.
\end{remark}

\begin{remark}
    When $n_a^++n_b^-=n_a^-+n_b^+$, then $\ker T_1\cong \ker\widetilde{T}_1$. Thus, by Theorem \ref{thm:well-posedness}(v) there exists a subspace $\lV$ of $\lW$ with $\lW_0\subseteq \lV$, such that $(T_1|_{\lV},\widetilde{T}_1|_{\lV})$ is a pair of mutually adjoint bijective realisations related to $(T_0,\widetilde{T}_0)$. 
    Another point of view is to say that in this case the skew-symmetric operator $T-\widetilde T$ (see \eqref{eq:TTtilda}) admits skew-selfadjoint extensions (cf.~\cite{ES23}). 
\end{remark}
\begin{remark}\label{rem:WisH1}
    If all the eigenvalues of $\mA(x)$ are strictly positive or strictly negative in $[a,b]$, then the graph space is $\lW=H^1((a,b);\C^r)$ and the minimal space is $\lW_0=H^1_0((a,b);\C^r)$. Using Theorem \ref{thm:codim}, $\dim\ker T_1+\dim \ker \widetilde{T}_1=r+r=2r$, which reveals a well-known fact that $H^1((a,b);\C^r)/H^1_0((a,b);\C^r)=2r$.
\end{remark}
\begin{remark}
    The result of Theorem \ref{thm:codim} is valid even if $\mA(x)$ is singular in some (or even all) points of the interval $[a,b]$. Thus we cover what might be called \emph{singular systems} of differential equations of the first order. These problems have a long history where often the problem of well-posedness was studied by obtaining the explicit formula in terms of (formal) series \cite[Chapter 4]{CO}
\end{remark}

\section{Examples}\label{sec:examples}
\subsection{Example 1}\label{ex:non-iso}
    Let us consider one-dimensional $(d=1)$ vectorial $(r=2)$ Friedrichs operator defined as follows
    \begin{align*}
        \mA(x)=\begin{bmatrix} 
	1 & 0  \\
	0 & 1-x \\
	\end{bmatrix},\  \mB(x)=\begin{bmatrix} 
	1 & 0  \\
	0 & 1 \\
	\end{bmatrix},\ x\in (0,1)\;.
    \end{align*}
    Here, $ \mB(x)+\mB^*(x)+\mA'(x)=\begin{bmatrix} 
	2 & 0  \\
	0 & 1 \\
	\end{bmatrix}$, hence conditions (F1) and (F2) are satisfied
 (see Example \ref{ex:cfo}).
 
 Due to the diagonal structure of this system, we can use the information obtained in $(r=1)$ case (see e.g.~\cite[Examples]{ES22}), to evaluate the dimensions of the kernels.
 Indeed, it is easy to see that $(\varphi_1,\varphi_2)^\top\in\ker T_1$
 if and only if for any $x\in (0,1)$ we have
 \begin{align*}
     \phi_1'(x)+\phi_1(x) =0 \quad \hbox{and} \quad
      ((1-x)\phi_2)'(x)+\phi_2(x) =0 \,.
 \end{align*}
 These scalar equations are already studied in this context in \cite{ES22} (see subsections 5.1 and 5.3 there) and the formal conclusion is that both equations contribute by 1 in the dimension of $\ker T_1$, i.e.
 $$
 \dim\ker T_1 = 2 \,.
 $$
 Analogously, for $\ker \widetilde T_1$ we need to study ($x\in (0,1)$)
  \begin{align*}
     -\phi_1'(x)+\phi_1(x) =0 \quad \hbox{and} \quad
      -((1-x)\phi_2)'(x)+\phi_2(x) =0 \,,
 \end{align*}
 which leads to 
  $$
 \dim\ker \widetilde T_1 = 1 \,.
 $$

Now we shall test the result of Theorem \ref{thm:codim} by investigating ranks of matrices
\begin{align*}
    \mA(0)=\begin{bmatrix}
            1&0\\0&1
        \end{bmatrix} \quad \hbox{and} \quad \mA(1)=\begin{bmatrix}
            1&0\\0&0
        \end{bmatrix}\;.
    \end{align*}
    Obviously, $n_0^+=2, n_0^-=0, n_1^+=1, n_1^-=0$. By Theorem \ref{thm:codim}, we have
    \begin{align*}
        \dim\ker T_1=2 \quad \hbox{and} \quad \dim\ker\widetilde{T}_1=1\;,
    \end{align*}
    which is in accordance with the information obtained as above.     

\begin{remark}
The choice of $\mB$ in this example is irrelevant.
More precisely, by Theorem \ref{thm:codim}, for any bounded $\mB$ such that the condition (F2) is satisfied (i.e.~the corresponding operators are Friedrichs operators) we have the same conclusion on the kernels.
Moreover, the same holds even for general operators (see \cite[Subsection 3.2]{ES23}).
\end{remark}
    
\subsection{Example 2: Second order linear ODE}
Let $I=(a,b)$ be an open interval. For $f\in L^2(I;\C)$, $p\in W^{1,\infty}(I; \R)$, $q\in L^{\infty}(I; \C)$, such that $p\geq \mu_0>0$, $\Re q\geq \mu_0>0$, consider the following ordinary differential equation on $I$,
\begin{equation}\label{eq:2ODE}
    -(p(x)u'(x))'+q(x)u(x)=f(x)\,.
\end{equation}

Let $\vu=(u_1,u_2)$ and consider the following system, for $\vf\in L^2(I;\C^2)$,
\begin{align*}
    T\vu:= (\mA\vu)'+\mB\vu = \vf\,,
\end{align*}
where 
\begin{align*}
    \mA= \begin{bmatrix}
        0 & -p \\ -p & 0
    \end{bmatrix} \quad \mathrm{and}\quad \mB = \begin{bmatrix}
        q & 0 \\ p'& p
    \end{bmatrix}.
\end{align*}
Here
\begin{align*}
    \mA'+\mB+\mB^*= \begin{bmatrix}
        2\Re q & 0 \\ 0 & 2p
    \end{bmatrix}\geq 2\mu_0\mI\;.
\end{align*}
For $\widetilde{T}\vu := -(\mA\vu)'+(\mA'+\mB^*)\vu$, the pair $(T, \widetilde T)$ forms a joint pair of abstract Friedrichs operators on $\lH= L^2(I; \C^2)$ with domain $\lD:= C_c^{\infty }(I; \C^2)$. For $ u_1 = u$, $u_2= u'$ and $\vf = (f,0)^\top$, the above system represents \eqref{eq:2ODE}. 
Therefore, one can easily transfer all well-posedness results for the system developed below to the original second-order equation \eqref{eq:2ODE}.

The graph space is $\lW=H^1(I;\C^2)$, which by the Sobolev embedding theorem means that for any $\vu\in \lW$ and $x\in [a,b]$, the evaluation $\vu(x)$ is well-defined. The boundary operator is given by
\begin{align*}
    (\forall \vu,\vv\in \lW)\quad \iscp{\vu}{\vv}=-p(b)u_2(b)\overline{v_1(b)}-p(b)u_1(b)\overline{v_2(b)}+p(a)u_2(a)\overline{v_1(a)}+p(a)u_1(a)\overline{v_2(a)}\,,
\end{align*}
and the minimal space $\lW_0$ is given by
\begin{align*}
    \lW_0= \{\vu\in \lW: u_1(a)=u_1(b)=u_2(a)=u_2(b)=0\}=H^1_0(I,\C^2)\;.
\end{align*}
Here the number of positive and negative eigenvalues are equal to 1 for both $\mA(b)$ and $\mA(a)$, i.e.~$n_a^+=n_a^-=n_b^+=n_b^-=1$ (hence we are in the regime discussed in Remark \ref{rem:WisH1}). Using Theorem \ref{thm:codim}, the dimensions of the kernels are equal to
\begin{align*}
    \dim\ker T_1=2 \quad \mathrm{and} \quad \dim\ker\widetilde{T}_1=2\;.
\end{align*}
Since both kernels are non-trivial, Theorem \ref{thm:well-posedness}(ii) guarantees the existence of infinitely many bijective realisations related to $(T, \widetilde T)$. The kernels are also isomorphic, thus by Theorem \ref{thm:well-posedness}(v), there exists a subspace $\lV$ with $\lW_0\subseteq \lV\subseteq \lW$ such that $(T_1|_{\lV}, \widetilde{T}_1|_{\lV})$ is pair of bijective realisations related to $(T,\widetilde T)$. In fact one such subspace is given by
\begin{align*}
    \lV_1= \{\vu \in \lW: u_1(a)=u_1(b)=0\}= H^1_0(I;\C)\times H^1(I;\C)\;.
\end{align*}
Due to Theorem \ref{thm:well-posedness}(i) it is enough to check that $(\lV_1,\lV_1)$ satisfies (V)-boundary conditions. For any $\vu, \vv\in \lV_1$ we clearly have $\iscp{\vu}{\vv}=0$, which in particular gives (V1)-boundary condition and $\lV_1\subseteq \lV_1^{[\perp]}$. Now for $\vv\in \lV_1^{[\perp]}$ and for any $\vu \in \lV_1 $ the expression $\iscp{\vu}{\vv}=0$ reads
\begin{align*}
    -p(b)u_2(b)\overline{v_1(b)}+p(a)u_2(a)\overline{v_1(a)}=0\;.
\end{align*}
Choice of $\vu\in \lV_1$, such that $u_2(a)\neq 0$, $u_2(b)= 0$, gives $v_1(a)=0$, and analogously for $v_1(b)$.  Thus, $\vv\in\lV_1$, implying $\lV_1^{[\perp]}\subseteq \lV_1$ and hence (V2)-boundary condition also holds. 

Note that the previous choice gives the homogeneous Dirichlet boundary condition for \eqref{eq:2ODE}. 
On the other hand the choice
\begin{align*}
    \lV_2=\widetilde \lV_2=\{\vu\in \lW: u_2(a)=u_2(b)=0\}= H^1(I;\C)\times H^1_0(I;\C)\,,
\end{align*}
gives the homogeneous Neumann boundary condition.

The restriction on the coefficient $q$ seems to be very strict, but thanks to the non-uniqueness of the representation of the equations as Friedrichs systems, we can ease this condition. For the choice $\vu=(u_1,u_2)^\top=(e^{-\beta x}u',e^{-\beta x}u)^\top$, for $\beta\in\R$, the equation \eqref{eq:2ODE} can be rewritten as
\begin{align*}
    (\mA\vu)'+\mB\vu = \vf\,,
\end{align*}
with
\begin{align*}
    \mA= \begin{bmatrix}
        p & 0 \\ 0 & p
    \end{bmatrix} \;, \quad 
    \mB = \begin{bmatrix}
        \beta p & -q \\ -p& \beta p-p'
    \end{bmatrix} \;, \quad
    \vf = \begin{bmatrix}
     e^{-\beta x}f\\ 0   
    \end{bmatrix} \;.
\end{align*}
Here
\begin{align*}
    \mA'+\mB+\mB^*= \begin{bmatrix}
        2\beta  p+ p' & -(p+q) \\ -(p+\overline{q}) & 2\beta  p-p'
    \end{bmatrix}\;,
\end{align*}
is positive for $p\geq \mu_0>0$, for any bounded $q$ and sufficiently large $\beta\in \R$. 

The graph space and the minimal space of the corresponding Friedrichs operator are again given by
$\lW=H^1(I;\C^2)$ and $\lW_0=H^1_0(I;\C^2)$,
but the boundary operator differs: $\vu,\vv\in \lW$,
\begin{align*}
    \iscp{\vu}{\vv}= p(b)u_1(b)\overline{v_1(b)}+p(b)u_2(b)\overline{v_2(b)}-p(a)u_1(a)\overline{v_1(a)}-p(b)u_2(a)\overline{v_2(a)}\,.
\end{align*}
$\mA(a)$ and $\mA(b)$ both have two positive eigenvalues and no negative eigenvalues. Hence, by Theorem \ref{thm:codim}, $\dim\ker T_1=2$ and $\dim\ker\widetilde{T}_1=2$. 

In this representation both the homogeneous Dirichlet and the homogeneous Neumann boundary conditions are still admissible (the realisations are bijective). 
However, here we cannot take the same boundary conditions for both operators $T$ and $\widetilde T$, as it was the case in the previous representation. Namely, both pairs $(\lV_1,\lV_2)$ and $(\lV_2,\lV_1)$ define bijective realisations, i.e.~if we take the homogeneous Dirichelt boundary condition for $T$, then the homogeneous Neumann boundary condition should be imposed for $\widetilde T$.

\section{Appendix}\label{sec:kato}
\subsection{Eigenvalues and eigenprojections}
In our setting, the eigenvalues and eigenprojections of a Lipschitz continuous matrix function $\mA\in W^{1,\infty}((a,b);\Mat{r})$ in the definition of Friedrichs operators (e.g.~ see Example \ref{ex:cfo}) are required to be Lipschitz continuous functions in some neighborhood of the end-points $a$ and $b$ of the interval $[a,b]$. The following theorem gives us the smoothness of the eigenvalues.

\begin{theorem}[Hoffman-Wielandt Inequality \cite{Bhatia}]\label{tm:hoff}
    Let $\mA$ and $\mB$ be Herimitian matrices of order $n$ and $\lambda_1(\mA)\geq ...\geq \lambda_n(\mA)$ and $\lambda_1(\mB)\geq ...\geq \lambda_n(\mB)$ be their eigenvalues, respectively. Then we have
    \begin{align*}
        \sum_{i=1}^n |\lambda_i(\mA)-\lambda_i(\mB)|^p\leq |\mA-\mB|_p^p\;.
    \end{align*}
    Where, $p\geq 1$ and $|\cdot |_p$ is the $p-$norm.
\end{theorem}

Using Theorem \ref{tm:hoff} we get that all eigenvalues of $\mA(x), x\in [a,b]$, are Lipschitz continuous functions of $\mA(x)$ and $\mA(x)$ is Lipschitz continuous in the interval $[a,b]$. Thus, all the eigenvalues are Lipschitz continuous on $[a,b]$. A similar result on eigenvectors is not expected. Here we present an argument on eigenprojections which resolves our specific situation.

We adapt our definitions and notations from \cite{Kato}.
\begin{itemize}
    \item[(i)] For the definition of \emph{$\lambda-$group}, we refer to Chapter II, Section 1.2 of \cite{Kato}. Let $\lambda_1(x)\geq \lambda_2(x)\geq,\dots, \lambda_r(x)$ be the eigenvalues of $\mA(x)$ (which are Lipschitz continuous). Note that we repeat eigenvalues with higher multiplicity (so-called repeated eigenvalues). Also, let $I(x)$ denotes the set of indices such that for any $x$ (close to $x_0$) $\{\lambda_i(x):i\in I(x)\}$ is the set of all eigenvalues of $\mA(x)$ forming the $\lambda-$group, where we are not repeating eigenvalues of higher multiplicity. Thus, $I(x)$ really depends on $x$.
    
    The \emph{total projection} corresponding to the $\lambda-$group is given by
    \begin{equation}\label{eq:totalproj_sum}
        \mP_{\lambda}(x)\;=\;\sum_{i\in I(x)} \mP_i(x)\,,
    \end{equation}
    where $\mP_i(x)$ denotes the eigenprojection corresponding to $\lambda_i(x)$. For any $x\in [a,b]$ it holds $\mP_i^2(x)=\mP_i(x)=\mP^*_i(x)$, while since $\mA(x)$ is Hermitian we have in addition
    \begin{equation*}
        \mP_i(x)\mP_j(x)=\delta_{i,j}\mP_j(x), \ i,j\in \{1,2,\dots,r\}
    \end{equation*}
    and
    \begin{equation*}
        \sum_{\lambda\in\sigma(\mA(x))} \mP_{\lambda}(x) = \mathbbm{1} \,.
    \end{equation*}
    
        Let us consider the following example to illustrate the definitions just introduced. Let $\mA(x)$ be a $2\times 2$ continuous function on $[0,1]$ such that it has eigenvalues $x$, $x\sin(1/x)$. Then we have
    $\lambda_1(x)=x$, $\lambda_2(x)=x\sin(1/x)$. At $x=0$, we have $\lambda_1(0)=\lambda_2(0)=0$ (the multiplicity is $2$), while for $\varepsilon>0$ (sufficiently small) and any $x\in (0,\eps)$ we have
    \begin{align*}
        I(x)\;=\;   \left\{
\begin{array}{ll}
      \{1\}, & (\exists k\in \mathbb{N})\ x=\frac{2}{k\pi} \,,\\
      \{1,2\}, & \mathrm{otherwise}\;.\\
\end{array} 
\right. 
    \end{align*}
    
    \item[(ii)] The resolvent definition of the \emph{total projection} can be found in \cite[Chapter II, Section 1.4]{Kato}. Let $x_0\in [a,b]$ and $\lambda $ be an eigenvalue of $\mA = \mA(x_0)$ with multiplicity $m$. Let $\Gamma$ be a positively oriented curve, say a circle, in the resolvent set $\rho(\mA)$ of $\mA$ enclosing $\lambda$ but no other eigenvalues of $\mA$. The second Neumann series is then convergent for $x$ sufficiently close to $x_0$, uniformly for $\xi\in \Gamma$. The existence of the resolvent $\mR(\xi,x)$ of $\mA(x)$ for $\xi\in\Gamma$ implies that there are no eigenvalues of $\mA(x)$ on $\Gamma$. The operator
    \begin{equation}\label{eq:totalproj}
        \mP_{\lambda}(x)\;=\;-\frac{1}{2\pi i}\int_{\Gamma} \mR(\xi,x)d\xi\,,
    \end{equation}
    is a projection and is equal to the sum of the eigenprojections of all the eigenvalues of $\mA(x)$ lying inside $\Gamma$. The eigenvalues of $\mA(x)$ lying inside $\Gamma$ form exactly the $\lambda-$group and thus $\mP_\lambda(x)$ defined in \eqref{eq:totalproj} corresponds to the total projection associated to the $\lambda-$group (see \eqref{eq:totalproj_sum}). Here $\mP(x_0)$ is precisely the eigenprojection corresponding to the eigenvalue $\lambda$ of $\mA$ and the following holds
    \begin{align*}
        \rank \mP_{\lambda}(x)=\rank \mP_{\lambda}(x_0)=m\,,
    \end{align*}
    for $x$ sufficiently close to $x_0$.

\end{itemize}

Let $x,y$ be sufficiently close to $x_0$ as previously. More precisely, let $\varepsilon>0$ be small and $x,y \in (x_0-\varepsilon, x_0+\varepsilon)$. For a fixed $\xi\in \Gamma$, using the resolvent identity
\begin{equation*}
    \mR(\xi,y)-\mR(\xi,x) = \mR(\xi,y)(\mA(x)-\mA(y))\mR(\xi,x)
\end{equation*}
we get
\begin{equation}\label{eq:resolvent}
\begin{aligned}
    |\mR(\xi,y)-\mR(\xi,x)|&\leq |\mR(\xi,y)||\mR(\xi,x)||\mA(y)-\mA(x)|\\
    &\leq\Bigl( \max_{\xi\in \Gamma,\, z\in (x_0-\varepsilon, x_0+\varepsilon)} |\mR(\xi,z)|\Bigr)^2|\mA(x)-\mA(y)|\,, 
\end{aligned}
\end{equation}
where $(\xi, x) \mapsto \mR(\xi, x)$ is continuous \cite[Section II.5.1]{Kato}. Hence, from \eqref{eq:resolvent}, we conclude that $\mR(\xi,x)$ is Lipschitz continuous in $x$ variable on $(x_0-\varepsilon,x_0+\varepsilon)$, uniformly in $\xi\in\Gamma$.
Therefore, from the resolvent definition of the total projection \eqref{eq:totalproj}, we get that $\mP_{\lambda}(x)$ is Lipschitz continuous on $(x_0-\varepsilon,x_0+\varepsilon)$.

\begin{remark}
    As mentioned in \cite[Chapter II, Section 5.3]{Kato}, in general, the eigenprojections do not have these continuity results, thus total projections can not be replaced by eigenprojections.
\end{remark}   

\subsection{Codimension}
In order to obtain the result of Theorem \ref{thm:codim}, we make use of the following elementary identities.
\begin{lemma}\label{lm:codimension1}\cite[Sec.~II.7, Proposition 6]{B74}
    {If $M$ and $N$ are two subspaces of a vector space $V$ of finite codimensions, then both $M+N$ and $M\cap N$ have finite codimensions and }
\begin{equation*}
    \mathrm{codim}(M+N)+\mathrm{codim}(M\cap N)=\mathrm{codim}(M)+\mathrm{codim}(N)\;.
\end{equation*}
\end{lemma} 
Previous theorem can be generalised as follows:
\begin{lemma}\label{lm:codimension2}
    If $\{M_k\}_{k=1,2,...,n}$ are subspaces of a vector space $V$ of finite codimensions, then
    \begin{align*}
        \mathrm{codim}{\bigcap_{k=1}^n}M_k= \sum_{k=1}^n\mathrm{codim}M_k+\sum_{j=2}^{n}(-1)^{n+j-1}\mathrm{codim} \bigl(\bigcap_{k=1}^{j-1}M_k+M_j\bigr)\;.
    \end{align*}
\end{lemma}
\section{Acknowledgements}
This work is supported by the Croatian Science Foundation under project 
IP-2022-10-7261 ADESO.

\end{document}